\documentclass[]{amsart}

\usepackage[utf8]{inputenc}

\usepackage{amsthm,amssymb}
\usepackage{mathtools}

\usepackage[english]{babel}

\DeclareMathOperator{\dist}{dist}

\allowdisplaybreaks[3]

\theoremstyle{plain}
\newtheorem{theorem}{Theorem}[section]
\newtheorem{proposition}[theorem]{Proposition}
\newtheorem{lemma}[theorem]{Lemma}
\newtheorem{corollary}[theorem]{Corollary}

\theoremstyle{definition}

\theoremstyle{remark}

\newcommand{\abs}[1]{\ensuremath\left\vert#1\right\vert}
\newcommand{\norm}[1]{\ensuremath\left\Vert#1\right\Vert}

\begin{document}

\title{Strongly zero product determined Banach algebras}

\author{J. Alaminos}
\address{Departamento de An\' alisis
Matem\' atico\\ Fa\-cul\-tad de Ciencias\\ Universidad de Granada\\
18071 Granada, Spain} 
\email{alaminos@ugr.es}
\author{J. Extremera}
\address{Departamento de An\' alisis
Matem\' atico\\ Fa\-cul\-tad de Ciencias\\ Universidad de Granada\\
18071 Granada, Spain}
\email{jlizana@ugr.es}
\author{M. L. C. Godoy}
\address{Departamento de An\' alisis
Matem\' atico\\ Fa\-cul\-tad de Ciencias\\ Universidad de Granada\\
18071 Granada, Spain} 
\email{mgodoy@ugr.es}
\author{A.\,R. Villena}
\address{Departamento de An\' alisis
Matem\' atico\\ Fa\-cul\-tad de
Ciencias\\ Universidad de Granada\\
18071 Granada, Spain} 
\email{avillena@ugr.es}

\begin{abstract}
   $C^*$-algebras, group algebras, and the algebra $\mathcal{A}(X)$ of approximable operators on a Banach space $X$ having the bounded approximation property are known to be zero product determined.
   We are interested in giving a quantitative estimate of this property by finding, for each Banach algebra $A$ of the above classes, a constant $\alpha$ with the property that for every continuous bilinear functional $\varphi\colon A \times A\to\mathbb{C}$ there exists a continuous linear functional $\xi$ on $A$ such that
   \[
      \sup_{\Vert a\Vert=\Vert b\Vert=1}\vert\varphi(a,b)-\xi(ab)\vert\le
      \alpha\sup_{\mathclap{\substack{\Vert a\Vert=\Vert b\Vert=1, \\ ab=0}}}\vert\varphi(a,b)\vert.
   \]
\end{abstract}


\thanks{
The authors were supported by MCIU/AEI/FEDER Grant PGC2018-093794-B-I00,
Junta de Andaluc\'{\i}a grant FQM-185, and
Proyectos I+D+i del programa operativo FEDER-Andaluc\'{\i}a  Grant A-FQM-48-UGR18.
The third named author was also supported by MIU PhD scholarship Grant FPU18/00419.
}

\maketitle

\section{Introduction}

Let $A$ be a Banach algebra. Then $\pi\colon A\times A\to A$ denotes the product map, we write $A^*$ for the dual of $A$, and $\mathcal{B}^2(A;\mathbb{C})$ for the space of continuous bilinear functionals on $A$.

The algebra $A$ is said to be \emph{zero product determined} if
every  $\varphi\in\mathcal{B}^2(A;\mathbb{C})$ with the property
\begin{equation}\label{15}
   a,b\in A, \ ab=0 \ \Rightarrow \ \varphi(a,b)=0
\end{equation}
belongs to the space
\[
   \mathcal{B}^2_\pi(A;\mathbb{C})=
   \bigl\{\xi\circ\pi : \xi\in A^*\bigr\}.
\]
This concept implicitly appeared in~\cite{ABEV0} as an additional outcome of the so-called property $\mathbb{B}$ which was introduced in that paper, and was the basis of subsequent Jordan and Lie versions (see~\cite{B1, B2, B3}). The algebra $A$ is said to have \emph{property} $\mathbb{B}$ if every $\varphi\in\mathcal{B}^2(A;\mathbb{C})$ satisfying~\eqref{15} belongs to the closed subspace $\mathcal{B}^2_b(A;\mathbb{C})$ of $\mathcal{B}^2(A;\mathbb{C})$ defined by
\[
   \mathcal{B}^2_b(A;\mathbb{C})=
   \bigl\{\psi\in\mathcal{B}^2(A;\mathbb{C}) : \psi(ab,c)=\psi(a,bc) \ \forall a,b,c\in A\bigr\}.
\]
In~\cite{ABEV0} it was shown that this class of Banach algebras is wide enough to include a number of examples of interest: $C^*$-algebras, the group algebra $L^1(G)$ of any locally compact group $G$, and the algebra $\mathcal{A}(X)$ of approximable operators on any Banach space $X$.

Throughout, we confine ourselves to Banach algebras having a bounded left approximate identity. Then $\mathcal{B}^2_\pi(A;\mathbb{C})=\mathcal{B}^2_b(A;\mathbb{C})$ (Proposition~\ref{913}), 
and hence $A$ is a zero product determined Banach algebra if and only if $A$ has property $\mathbb{B}$. For example, this applies to $C^*$-algebras, group algebras and the algebra $\mathcal{A}(X)$
on any Banach space $X$ having the bounded approximation property,
so that all of them are zero product determined.

For each $\varphi\in\mathcal{B}^2(A;\mathbb{C})$,
the distance from $\varphi$ to $\mathcal{B}^2_\pi(A;\mathbb{C})$ is
\[
   \dist\left(\varphi,\mathcal{B}^2_\pi(A;\mathbb{C})\right)=
   \inf\left\{\norm{\varphi - \psi} :  \psi\in\mathcal{B}^2_\pi(A;\mathbb{C})\right\},
\]
which can be easily estimated through the constant
\begin{equation*}
   \abs{\varphi}_b =
   \sup\left\{\abs{\varphi(ab,c)-\varphi(a,bc)} :  a,b,c\in A, \ \Vert a\Vert=\Vert b\Vert=\Vert c\Vert=1\right\}
\end{equation*}
(Proposition~\ref{913} below).
Our purpose is to estimate $\dist \left(\varphi,\mathcal{B}^2_\pi(A;\mathbb{C})\right)$ through the constant
\[
   \abs{\varphi}_{zp}=
   \sup\left\{\abs{\varphi(a,b)} :  a,b\in A, \ \norm{a} = \norm{b} = 1, \ ab=0\right\}.
\]
Note that $A$ is zero product determined precisely when
\begin{equation}\label{21}
   \varphi\in\mathcal{B}^2(A;\mathbb{C}), \ \abs{\varphi}_{zp}=0 \ \Rightarrow \ \varphi\in\mathcal{B}^2_\pi(A;\mathbb{C}).
\end{equation}
We call $A$  \emph{strongly zero product determined} if condition~\eqref{21} is strengthened by requiring that there is a distance estimate
\begin{equation}\label{22}
   \dist \bigl(\varphi,\mathcal{B}^2_\pi(A;\mathbb{C})\bigr)\le
   \alpha \abs{\varphi}_{zp}\quad\forall \varphi\in\mathcal{B}^2(A;\mathbb{C})
\end{equation}
for some constant $\alpha$; in this case, the optimal constant $\alpha$ for which~\eqref{22} holds will be denoted by $\alpha_A$.
The inequality $\abs{\varphi}_{zp} \le \dist \left(\varphi,\mathcal{B}^2_\pi(A;\mathbb{C})\right)$ is always true (Proposition~\ref{913} below).
We also note that $A$ has property $\mathbb{B}$ exactly in the case when
\begin{equation}\label{23}
   \varphi\in\mathcal{B}^2(A;\mathbb{C}), \ \abs{\varphi}_{zp}=0 \ \Rightarrow \ \abs{\varphi}_b=0,
\end{equation}
and the algebra $A$ is said to have the \emph{strong property} $\mathbb{B}$ if there is an estimate
\begin{equation}\label{24}
   \vert\varphi\vert_{b} \le \beta \abs{\varphi}_{zp}\quad \forall \varphi\in\mathcal{B}^2(A;\mathbb{C})
\end{equation}
for some constant $\beta$; in this case, the optimal constant $\beta$ for which~\eqref{24} holds will be denoted by $\beta_A$.
The inequality $\abs{\varphi}_{zp}\le M \abs{\varphi}_b$ is always true for some constant~$M$ (Proposition~\ref{913} below).
The spirit of this concept first appeared in~\cite{Vil}, and was subsequently formulated in~\cite{SS2, SS}.
This property has proven to be useful to study the hyperreflexivity of the spaces of derivations and continuous cocycles on $A$ 
(see~\cite{SS2, SS, Vil1, Vil2, Sam1}).

From~\cite[Corollary~1.3]{Vil3}, it follows that each $C^*$-algebra $A$ is strongly zero product determined, has the strong property $\mathbb{B}$, and $\alpha_A,\beta_A\le 8$.
It is shown in~\cite{SS} that each group algebra has the strong property $\mathbb{B}$
and so (by Corollary~\ref{914} below) it is also strongly zero product determined.
In this paper we give, for each group algebra, a sharper estimate for the constant of the strong property $\mathbb{B}$ to the one given in~\cite[Theorem~3.4]{SS}.
Finally, we prove that the algebra $\mathcal{A}(X)$ is strongly zero product determined for each Banach space $X$ having property $(\mathbb{A})$
(which is a rather strong approximation property for the space $X$).

Throughout, our reference for Banach algebras, and particularly for group algebras, is the monograph~\cite{D}.

\section{Elementary estimates}

In this section, we gather together some estimates that relate the seminorms
$\dist \left( \cdot ,\mathcal{B}^2_\pi(A;\mathbb{C})\right)$,
$\abs{\cdot}_{b}$, and $\abs{\cdot}_{zp}$ on $\mathcal{B}^2_\pi(A;\mathbb{C})$
to each other.

\begin{proposition}\label{913}
   Let $A$ be a Banach algebra with a left approximate identity of bound $M$.
   Then
   $\mathcal{B}^2_\pi(A;\mathbb{C})=\mathcal{B}^2_b(A;\mathbb{C})$
   and, for each $\varphi\in\mathcal{B}^2(A;\mathbb{C})$, the following properties hold:
   \begin{enumerate}
      \item[\rm (i)]
         The distance $\dist \left(\varphi,\mathcal{B}^2_\pi(A;\mathbb{C})\right)$
         is attained;
      \item[\rm (ii)]
         $\frac{1}{2} \abs{\varphi}_b\le
            \dist \left(\varphi,\mathcal{B}^2_\pi(A;\mathbb{C})\right)\le
            M \abs{\varphi}_b$;
      \item[\rm (iii)]
         $\abs{\varphi}_{zp} \le \dist \left(\varphi,\mathcal{B}^2_\pi(A;\mathbb{C})\right)$.
   \end{enumerate}
\end{proposition}

\begin{proof}
   Let ${(e_\lambda)}_{\lambda\in\Lambda}$ be a left approximate identity of bound $M$.

   (i)
   Let $(\xi_n)$ be a sequence in $A^*$ such that
   \[
      \dist \left(\varphi,\mathcal{B}^2_\pi(A;\mathbb{C})\right) =
      \lim_{n\to\infty}\norm{\varphi-\xi_n\circ\pi}.
   \]
   For each $n\in\mathbb{N}$ and $a\in A$, we have
   \[
      \abs{\xi_n(e_\lambda a)} = 
      \abs{(\xi_n\circ\pi)(e_\lambda,a)} \le
      M \norm{\xi_n\circ\pi} \norm{a} \quad \forall \lambda\in\Lambda
   \]
   and hence, taking limit in the above inequality  and using that $\lim_{\lambda\in\Lambda}e_\lambda a=a$, we see that
   $\abs{\xi_n(a)} \le M \norm{\xi_n\circ\pi} \norm{a}$,
   which shows that $\norm{\xi_n} \le M \norm{\xi_n\circ\pi}$.
   Further, since
   \[
      \Vert\xi_n\circ\pi\Vert\le\Vert\varphi-\xi_n\circ\pi\Vert+\Vert\varphi\Vert\quad\forall n\in\mathbb{N},
   \]
   it follows that the sequence $(\Vert\xi_n\Vert)$ is bounded.
   By the Banach--Alaoglu theorem, the sequence $(\xi_n)$ has a weak$^*$-accumulation point, say $\xi$, in $A^*$.
   Let ${(\xi_\nu)}_{\nu\in N}$ be a subnet of $(\xi_n)$ such that $\text{w}^*$-$\lim_{\nu\in N}\xi_\nu=\xi$.
   The task is now to show that
   \[
      \norm{\varphi-\xi\circ\pi} = \dist \left(\varphi,\mathcal{B}^2_\pi(A;\mathbb{C})\right).
   \]
   For each $a,b\in A$ with $\Vert a\Vert=\Vert b\Vert=1$, we have
   \[
      \abs{\varphi(a,b)-\xi_\nu(ab)} \le \norm{\varphi-\xi_\nu\circ\pi} \quad\forall \nu\in N,
   \]
   and so, taking limits on both sides of the above inequality and using
   that
   \[
      \lim_{\nu\in N}\xi_\nu(ab)=\xi(ab)
   \]
   and that $\left(\Vert\varphi-\xi_\nu\circ\pi\Vert\right)_{\nu\in N}$ is a subnet of the convergent sequence $(\Vert\varphi-\xi_n\circ\pi\Vert)$,
   we obtain
   \[
      \vert\varphi(a,b)-\xi(ab)\vert \le \dist \left(\varphi,\mathcal{B}^2_\pi(A;\mathbb{C})\right).
   \]
   This implies that $\norm{\varphi-\xi\circ\pi} \le \dist \left(\varphi,\mathcal{B}^2_\pi(A;\mathbb{C})\right)$, and the converse inequality
   $\dist \left(\varphi,\mathcal{B}^2_\pi(A;\mathbb{C})\right) \le \norm{\varphi-\xi\circ\pi}$
   trivially holds.

   (ii)
   For each $\lambda\in\Lambda$ define $\xi_\lambda\in A^*$ by
   \[
      \xi_\lambda(a)=\varphi(e_\lambda,a)\quad\forall a\in A.
   \]
   Then $\norm{\xi_\lambda} \le M \norm{\varphi}$ for each $\lambda\in\Lambda$, so that
   ${(\xi_\lambda)}_{\lambda\in\Lambda}$ is a bounded net in $A^*$ and hence the
   Banach--Alaoglu theorem shows that it has a weak$^*$-accumulation point, say $\xi$, in $A^*$.
   Let ${(\xi_\nu)}_{\nu\in N}$ be a subnet of $(\xi_\lambda)_{\lambda\in\Lambda}$ such that
   $\text{w}^*$-$\lim_{\nu\in N}\xi_\nu=\xi$.
   For each $a,b\in A$ with $\Vert a\Vert=\Vert b\Vert=1$, we have
   \[
      \abs{\varphi(e_\nu a,b)-\varphi(e_\nu,ab)} \le M \abs{\varphi}_b 
      \quad\forall \nu\in N
   \]
   and hence, taking limit and using that ${(e_\nu a)}_{\nu\in N}$ is a subnet of the
   convergent net ${(e_\lambda a)}_{\lambda\in\Lambda}$ and that
   $\lim_{\nu\in N}\varphi(e_\lambda,ab)=\xi(ab)$, we see that
   \[
      \abs{\varphi(a,b)-\xi(ab)} \le M \abs{\varphi}_b.
   \]
   This gives $\Vert\varphi-\xi\circ\pi\Vert\le M\abs{\varphi}_b$,
   whence
   \[
      \dist \left(\varphi,\mathcal{B}^2_\pi(A;\mathbb{C})\right)\le M\vert\varphi\vert_b.
   \]

   Set $\xi\in A^*$.
   For each $a,b,c\in A$ with $\Vert a\Vert=\Vert b\Vert=\Vert c\Vert=1$, we have
   \begin{equation*}
      \begin{split}
         \abs{\varphi(ab,c)-\varphi(a,bc)} &
         = \abs{\varphi(ab,c)-(\xi\circ\pi)(ab,c)+(\xi\circ\pi)(a,bc)-\varphi(a,bc)} \\
         & \le \abs{\varphi(ab,c)-(\xi\circ\pi)(ab,c)}+\abs{(\xi\circ\pi)(a,bc)-\varphi(a,bc)} \\
         & \le \norm{\varphi-\xi\circ\pi} \norm{ab} \norm{c} +
         \norm{\varphi-\xi\circ\pi} \norm{a} \norm{bc} \\
         & \le 2 \norm{\varphi-\xi\circ\pi}
      \end{split}
   \end{equation*}
   and therefore
   $\abs{\varphi}_b \le 2 \norm{\varphi-\xi\circ\pi}$.
   Since this inequality holds for each $\xi\in A^*$, it follows that
   \[
      \abs{\varphi}_b\le 2\dist \left(\varphi,\mathcal{B}^2_\pi(A;\mathbb{C})\right).
   \]

   (iii)
   Let $a,b\in A$ with $\Vert a\Vert=\Vert b\Vert=1$ and $ab=0$.
   For each $\xi\in A^*$,  we see that
   \[
      \abs{\varphi(a,b)} = \abs{\varphi(a,b)-(\xi\circ\pi)(a,b)} \le
      \norm{\varphi-\xi\circ\pi},
   \]
   and consequently
   $\abs{\varphi}_{zp} \le \norm{\varphi-\xi\circ\pi}$. 
   Since the above inequality holds for each $\xi\in A^*$, we conclude that
   \[
      \abs{\varphi}_{zp}
      \le
      \dist \left(\varphi,\mathcal{B}^2_\pi(A;\mathbb{C})\right).
   \]

Finally, it is clear that $\mathcal{B}^2_\pi(A;\mathbb{C})\subset\mathcal{B}^2_b(A;\mathbb{C})$. 
To prove the reverse inclusion take $\varphi\in\mathcal{B}^2_b(A;\mathbb{C})$.  Then $\vert\varphi\vert_b=0$, 
hence (ii) shows that $\dist \left(\varphi,\mathcal{B}^2_\pi(A;\mathbb{C})\right) = 0$, and (i) gives 
$\psi\in\mathcal{B}^2_\pi(A;\mathbb{C})$ such that  $\Vert\varphi-\psi\Vert=0$, which implies that 
$\varphi=\psi\in\mathcal{B}^2_\pi(A;\mathbb{C})$. 
\end{proof}

The following result is an immediate consequence of assertion (ii) in Proposition~\ref{913}.

\begin{corollary}\label{914}
   Let $A$ be a Banach algebra with a left approximate identity of bound~$M$.
   Then $A$ is strongly zero product determined if and only if has the strong property $\mathbb{B}$, in which case
   \[
      \tfrac{1}{2}\beta_A\le\alpha_A\le M\beta_A.
   \]
\end{corollary}

\section{Group algebras}

In this section we prove that the group algebra $L^1(G)$ of each locally compact group $G$ is strongly zero product determined and we provide an estimate of the constants $\alpha_{L^1(G)}$ and $\beta_{L^1(G)}$.
Our estimate of $\beta_{L^1(G)}$ improves the one given in~\cite{SS}.
For the basic properties of this important class of Banach algebras we refer the reader to \cite[Section~3.3]{D}.

Throughout this section, $\mathbb{T}$ denotes the circle group, and we consider the normalized Haar measure on $\mathbb{T}$. We write $A(\mathbb{T})$ and $A(\mathbb{T}^2)$ for the Fourier algebras of $\mathbb{T}$ and $\mathbb{T}^2$, respectively. For each $f\in A(\mathbb{T})$, $F\in A(\mathbb{T}^2)$, and $j,k\in\mathbb{Z}$,
we write $\widehat{f}(j)$ and $\widehat{F}(j,k)$ for the Fourier coefficients of $f$ and $F$, respectively.
Let $\mathbf{1},\zeta\in A(\mathbb{T})$ denote the functions defined by
\[
   \mathbf{1}(z)=1,\quad 
   \zeta(z)=z\quad\forall z\in\mathbb{T}.
\]
Let $\Delta\colon A(\mathbb{T}^2)\to A(\mathbb{T})$ be the bounded linear map defined by
\[
   \Delta(F)(z)=F(z,z) \quad
   \forall z\in\mathbb{T}, \ \forall F\in A(\mathbb{T}^2).
\]
For $f,g\in A(\mathbb{T})$, let $f\otimes g\colon \mathbb{T}^2\to\mathbb{C}$
denote the function defined by
\[
   (f\otimes g)(z,w)=f(z)g(w)\quad \forall z,w\in \mathbb{T},
\]
which is an element of $A(\mathbb{T}^2)$ with $\norm{f\otimes g} = \norm{f} \norm{g}$.

\begin{lemma}\label{l1}
   Let $\Phi\colon A(\mathbb{T}^2)\to\mathbb{C}$ be a continuous linear functional, and
   let the constant $\varepsilon\ge 0$ be such that
   \begin{equation*}
      f,g\in A(\mathbb{T}), \ fg=0 \ \Rightarrow \
      \abs{\Phi(f\otimes g)} \le \varepsilon \norm{f} \norm{g}.
   \end{equation*}
   Then
   \begin{equation*}
      \abs{\Phi(\zeta\otimes\mathbf{1}-\mathbf{1}\otimes\zeta)} \le
      \norm{\Phi\mid_{\ker\Delta}} 2\sin\tfrac{\pi}{10}+60\sqrt{27} \left(1+\sin\tfrac{\pi}{10}\right) \varepsilon.
   \end{equation*}
\end{lemma}

\begin{proof}
   Set
   \begin{equation*}
   	\begin{split}
		E & =\left\{e^{\theta i} : -\tfrac{1}{5}\pi\le \theta\le\tfrac{1}{5}\pi\right\}, \\
	    W & = \left\{(z,w)\in\mathbb{T}^2 : zw^{-1}\in E\right\},
	 \end{split}
	\end{equation*}
   and let $F\in A(\mathbb{T}^2)$ be such that
   \begin{equation}\label{1505}
      F(z,w)=0\quad\forall (z,w)\in W.
   \end{equation}
   Our objective is to prove that
   \begin{equation}\label{e1426}
      \abs{\Phi(F)}
      \le 
      30\sqrt{27} \norm{F} \varepsilon.
   \end{equation}
   For this purpose, we take
   \begin{equation*}
      \begin{split}
         a & =e^{\frac{1}{15}\pi\, i},\\
         A & =\left\{e^{\theta i} : 0<\theta\le\tfrac{1}{15}\pi\right\},\\
         B & =  \left\{ e^{\theta i} : \tfrac{2}{15}\pi<\theta\le\tfrac{29}{15}\pi\right\}, \\
         U & =  \left\{e^{\theta i} :  -\tfrac{1}{30}\pi<\theta<\tfrac{1}{30}\pi\right\},
      \end{split}
   \end{equation*}
   and we define functions $\omega,\upsilon\in A(\mathbb{T})$ by
   \[
      \omega=
      30 \, \chi_A\ast\chi_U, \quad  \upsilon=
      30 \, \chi_B\ast\chi_U.
   \]
   We note that
   \begin{gather*}
      \left\{z\in\mathbb{T} :  \omega(z)\ne 0\right\} =
      AU=
      \left\{e^{\theta i} : -\tfrac{1}{30}\pi<\theta<\tfrac{1}{10}\pi\right\},
      \\
      \left\{z\in\mathbb{T} :  \upsilon(z)\ne 0\right\} =
      BU=
      \left\{e^{\theta i} :  \tfrac{1}{10}\pi<\theta<\tfrac{59}{30}\pi\right\},
   \end{gather*}
   and, with $\norm{\cdot}_2$ denoting the norm of $L^2(\mathbb{T})$,
   \begin{align*}
      \Vert\omega\Vert & \le
      30 \norm{\chi_A}_2 \norm{\chi_U}_2 = 30\,\frac{1}{\sqrt{30}}\,\frac{1}{\sqrt{30}}=1,
      \\
      \Vert\upsilon\Vert & \le
      30 \norm{\chi_B}_2 \norm{\chi_U}_2 = 30\,\frac{\sqrt{27}}{\sqrt{30}}\,\frac{1}{\sqrt{30}}=\sqrt{27}.
   \end{align*}
   Since
   \[
      \bigcup_{k=0}^{29} a^k A = \mathbb{T},
      \quad
      \bigcup_{k=2}^{28} a^k A = B,
   \]
   it follows that
   \[
      \sum_{k=0}^{29}\delta_{a^k}\ast\chi_{A}=
      \sum_{k=0}^{29}\chi_{a^k A}=
      \mathbf{1},
      \quad
      \sum_{k=2}^{28}\delta_{a^k}\ast\chi_{A}=
      \sum_{k=2}^{28}\chi_{a^k A}=
      \chi_B,
   \]
   and thus,
   for each $j\in\mathbb{Z}$, we have
   \begin{align}\label{e1905}
      \sum_{k=j}^{j+29}\delta_{a^k}\ast\omega   & =
      30\delta_{a^j}\ast\sum_{k=0}^{29}\delta_{a^k}\ast\chi_A\ast\chi_U=
      30\delta_{a^j}\ast\mathbf{1}\ast\chi_U=\mathbf{1}, \\
      \label{e1906}
      \sum_{k=j+2}^{j+28}\delta_{a^k}\ast\omega & =
      30\delta_{a^j}\ast\sum_{k=2}^{28}
      \delta_{a^k}\ast\chi_A\ast\chi_U=
      30\delta_{a^j}\ast\chi_B\ast\chi_U=
      \delta_{a^j}\ast\upsilon.
   \end{align}
   If $j\in\mathbb{Z}$, $k\in\{j-1,j,j+1\}$, and $z,w\in\mathbb{T}$ are such that
   $(\delta_{a^j}\ast\omega)(z)(\delta_{a^k}\ast\omega)(w)\ne 0$, then
   \[
      zw^{-1}\in a^j AU \bigl(a^k AU\bigr)^{-1}
      \subset
      a^{j-k}\left\{e^{\theta i} : -\tfrac{2}{15}\pi<\theta<\tfrac{2}{15}\pi\right\}
      \subset
      E,
   \]
   whence
   $\left\{(z,w)\in\mathbb{T}^2 : (\delta_{a^j}\ast\omega)\otimes(\delta_{a^k}\ast\omega)(z,w)\ne 0\right\}
      \subset W$
   and~\eqref{1505} gives
   \begin{equation}\label{1508}
      F(\delta_{a^j}\ast\omega)\otimes(\delta_{a^k}\ast\omega)=0.
   \end{equation}
   Since $AU \cap  BU=\varnothing$,
   it follows that $\omega\upsilon=0$, and therefore
   \begin{equation}\label{1745}
      (\delta_{a^k}\ast\omega)(\delta_{a^k}\ast\upsilon)=0\quad\forall k\in\mathbb{Z}.
   \end{equation}
   From \eqref{e1905}, \eqref{e1906}, and \eqref{1508} we deduce that
   \begin{equation*}
      \begin{split}
         F & =
         F\sum_{j=0}^{29}\sum_{k=j-1}^{j+28}
         (\delta_{a^j}\ast\omega)\otimes(\delta_{a^k}\ast\omega)\\
         & =  \sum_{j=0}^{29}\sum_{k=j-1}^{j+1}
         F(\delta_{a^j}\ast\omega)\otimes(\delta_{a^k}\ast\omega)
         +
         \sum_{j=0}^{29}\sum_{k=j+2}^{j+28}
         F(\delta_{a^j}\ast\omega)\otimes(\delta_{a^k}\ast\omega) \\
         & =
         \sum_{j=0}^{29}\sum_{k=j+2}^{j+28}
         F(\delta_{a^j}\ast\omega)\otimes(\delta_{a^k}\ast\omega)
         =
         \sum_{j=0}^{29}
         F(\delta_{a^j}\ast\omega)\otimes(\delta_{a^{j}}\ast\upsilon)
         .
      \end{split}
   \end{equation*}
   As
   \[
      F=\sum_{j,k=-\infty}^\infty\widehat{F}(j,k)\zeta^j\otimes \zeta^k
   \]
   we have
   \[
      F=\sum_{j,k=-\infty}^\infty\sum_{l=0}^{29}\widehat{F}(j,k)
      \bigl(\zeta^j(\delta_{a^l}\ast\omega)\bigr)\otimes\bigl(\zeta^{k}(\delta_{a^{l}}\ast\upsilon)\bigr),
   \]
   so that
   \[
      \Phi(F)=
      \sum_{j,k=-\infty}^\infty\sum_{l=0}^{29}\widehat{F}(j,k)
      \Phi\Bigl(\bigl(\zeta^j(\delta_{a^l}\ast\omega)\bigr)\otimes\bigl(\zeta^{k}(\delta_{a^{l}}\ast\upsilon)\bigr)\Bigr).
   \]
   By \eqref{1745}, for each $j,k,l\in\mathbb{Z}$,
   \[
      \bigl(\zeta^j(\delta_{a^l}\ast\omega)\bigr)\bigl(\zeta^{k}(\delta_{a^{l}}\ast\upsilon)\bigr)=0
   \]
   and therefore
   \begin{equation*}
      \begin{split}
         \abs{\Phi\left(
         \bigl(\zeta^j(\delta_{a^l}\ast\omega)\bigr)\otimes\bigl(\zeta^{k}(\delta_{a^{l}}\ast\upsilon)\bigr)\right)}
         &\le
         \varepsilon
         \norm{\zeta^j(\delta_{a^l}\ast\omega)} 
         \norm{\zeta^k(\delta_{a^{l}}\ast\upsilon)}\\
         & =
         \varepsilon \norm{\omega} \norm{\upsilon} \le 
         \sqrt{27}\,\varepsilon.
      \end{split}
   \end{equation*}
   We thus get
   \begin{equation*}
      \begin{split}
         \abs{\Phi(F)}
         & =
         \sum_{j,k=-\infty}^\infty \sum_{l=0}^{29} \abs{\widehat{F}(j,k)} 
         \abs{\Phi \Bigl(\bigl(\zeta^j(\delta_{a^l}\ast\omega)\bigr)\otimes
         \bigl(\zeta^{k}(\delta_{a^{l}}\ast\upsilon)\bigr) \Bigr)} \\
         &\le
         \sum_{j,k=-\infty}^\infty\sum_{l=0}^{29} \abs{\widehat{F}(j,k)}
         \sqrt{27}\,\varepsilon = 30\sqrt{27} \norm{F} \varepsilon
         ,
      \end{split}
   \end{equation*}
   and~\eqref{e1426} is proved.

   Let $f\in A(\mathbb{T})$ be such that $f(z)=0$ for each $z\in E$,
   and define the function $F\colon\mathbb{T}^2\to\mathbb{C}$ by
   \[
      F(z,w)=f(zw^{-1})w
      =\sum_{k=-\infty}^{\infty}\widehat{f}(k)z^kw^{-k+1}\quad\forall z,w\in\mathbb{T}.
   \]
   Then
   $F\in A(\mathbb{T}^2)$,
   $\Vert F\Vert=\Vert f\Vert$,
   $\zeta\otimes\mathbf{1}-\mathbf{1}\otimes\zeta-F\in\ker\Delta$,
   and
   \[
      \left(\zeta\otimes\mathbf{1}-\mathbf{1}\otimes\zeta-F\right)(z,w) =
      \left(1-\widehat{f}(1) \right)z + \left(-1-\widehat{f}(0) \right)w -
      \sum_{k\ne 0,1}\widehat{f}(k)z^kw^{-k+1},
   \]
   which certainly implies that
   \[
      \norm{\zeta\otimes\mathbf{1}-\mathbf{1}\otimes\zeta-F} =
      \abs{ 1-\widehat{f}(1)}
      +
      \abs{-1-\widehat{f}(0)}
      +
      \sum_{k\ne 0,1}\abs{\widehat{f}(k)}
      =
      \norm{\zeta-\mathbf{1}-f}.
   \]
   According to~\eqref{e1426}, we have
   \begin{equation*}
      \begin{split}
         \abs{\Phi(\zeta\otimes\mathbf{1}-\mathbf{1}\otimes\zeta)}
         & \le
         \abs{\Phi(\zeta\otimes\mathbf{1}-\mathbf{1}\otimes\zeta-F)}
         +
         \abs{\Phi(F)}\\
         &\le
         \norm{\Phi\mid_{\ker\Delta}} \norm{\zeta\otimes\mathbf{1}-\mathbf{1}\otimes\zeta-F} +
         30\sqrt{27} \norm{F} \varepsilon\\
         & =
         \norm{\Phi\mid_{\ker\Delta}} \norm{\zeta-\mathbf{1}-f} + 30\sqrt{27}\norm{f} \varepsilon\\
         & \le
         \norm{\Phi\mid_{\ker\Delta}} \norm{\zeta-\mathbf{1}-f} + 30\sqrt{27}\bigl(\Vert\zeta-\mathbf{1}-f\Vert+2\bigr)\varepsilon
      \end{split}
   \end{equation*}
   (as $\Vert f\Vert\le\Vert\zeta-\mathbf{1}-f\Vert+\Vert\zeta-\mathbf{1}\Vert$).
   Further, this inequality holds for each function from the set $\mathcal{I}$
   consisting of all functions $f\in A(\mathbb{T})$ such that $f(z)=0$
   for each $z\in E$. Consequently,
   \begin{equation*}
      \abs{\Phi(\zeta\otimes\mathbf{1}-\mathbf{1}\otimes\zeta)} \le
      \norm{\Phi\mid_{\ker\Delta}} \dist (\zeta-\mathbf{1},\mathcal{I}) +
      30\sqrt{27}\left( \dist (\zeta-\mathbf{1},\mathcal{I})+2\right)\varepsilon.
   \end{equation*}
   On the other hand, it is shown at the beginning of the proof of \cite[Corollary~3.3]{AR} that
   \[
      \dist (\zeta-\mathbf{1},\mathcal{I})\le 2\sin\tfrac{\pi}{10},
   \]
   and we thus get
   \begin{equation*}
      \abs{\Phi(\zeta\otimes\mathbf{1}-\mathbf{1}\otimes\zeta)} \le
      \norm{\Phi\mid_{\ker\Delta}} 2\sin\tfrac{\pi}{10}
      +
      30\sqrt{27} \left(2\sin\tfrac{\pi}{10}+2\right)\varepsilon,
   \end{equation*}
   which completes the proof.
\end{proof}

\begin{lemma}\label{l2}
   Let $\Phi\colon A(\mathbb{T}^2)\to\mathbb{C}$ be a continuous linear functional, and let the constant $\varepsilon\ge 0$ be such that
   \begin{equation*}
      f,g\in A(\mathbb{T}), \ fg=0 \ \Rightarrow \
      \abs{\Phi(f\otimes g)} \le \varepsilon \norm{f} \norm{g}.
   \end{equation*}
   Then
   \begin{equation*}
      \abs{\Phi\bigl(F-\mathbf{1}\otimes\Delta F\bigr)} \le
      60\sqrt{27}\,\frac{1+\sin\tfrac{\pi}{10}}{1-2\sin\tfrac{\pi}{10}}\,\varepsilon \norm{F}
   \end{equation*}
   for each $F\in A(\mathbb{T}^2)$.
\end{lemma}

\begin{proof}
   Fix $j,k\in\mathbb{Z}$. We claim that
   \begin{equation}\label{e937}
      \abs{\Phi(\zeta^j\otimes\zeta^k-\mathbf{1}\otimes\zeta^{j+k})}
      \le
      \norm{\Phi\mid_{\ker\Delta}} 2\sin\tfrac{\pi}{10}+
      60\sqrt{27}\left(1+\sin\tfrac{\pi}{10}\right)\varepsilon.
   \end{equation}
   Of course, we are reduced to proving~\eqref{e937} for $j\ne 0$.
   We define
   $d_j\colon A(\mathbb{T})\to A(\mathbb{T})$, and
   $D_j,L_k\colon A(\mathbb{T}^2)\to A(\mathbb{T}^2)$
   by
   \[
      d_jf(z)=f(z^j)\quad\forall f\in A(\mathbb{T}), \ \forall z\in\mathbb{T}
   \]
   and
   \[
      D_jF(z,w)=F(z^j,w^j), \quad
      L_kF(z,w)=F(z,w)w^k
      \quad\forall F\in A(\mathbb{T}^2), \ \forall z,w\in\mathbb{T},
   \]
   respectively.
   Further, we consider the continuous linear functional
   $\Phi\circ L_k\circ D_j$.
   If $f,g\in A(\mathbb{T})$ are such that $fg=0$, then
   $(d_jf)(\zeta^kd_jg)=\zeta^{k}d_j(fg)=0$,
   and so, by hypothesis,
   \[
      \abs{\Phi\circ L_k\circ D_j(f\otimes g)} =
      \abs{\Phi(d_j f\otimes\zeta^k d_j g)} 
      \le
      \varepsilon \norm{d_j f} \norm{\zeta^k d_j g} =
      \varepsilon \norm{f} \norm{g}.
   \]
   By applying Lemma~\ref{l1}, we obtain
   \begin{equation*}
      \begin{split}
         \abs{\Phi(\zeta^j\otimes\zeta^k-\mathbf{1}\otimes\zeta^{j+k})}
         &=
         \abs{\Phi\circ L_k\circ D_j(\zeta\otimes\mathbf{1}-\mathbf{1}\otimes\zeta)} \\
         &\le
         \norm{\Phi\circ L_k\circ D_j\mid_{\ker\Delta}} 2\sin\tfrac{\pi}{10} +
         60\sqrt{27} \left(1+\sin\tfrac{\pi}{10}\right)\varepsilon.
      \end{split}
   \end{equation*}
   We check at once that $(L_k\circ D_j)(\ker\Delta)\subset\ker\Delta$, which gives
   \[
      \norm{\Phi\circ L_k\circ D_j\mid_{\ker\Delta}} \le \norm{\Phi\mid_{\ker\Delta}},
   \]
   and therefore~\eqref{e937} is proved.

   Take $F\in A(\mathbb{T}^2)$. Then
   \begin{align*}
      F        & =\sum_{j,k=-\infty}^\infty\widehat{F}(j,k)\zeta^j\otimes\zeta^k \\
      \intertext{and}
      \Delta F & =\sum_{j,k=-\infty}^\infty\widehat{F}(j,k)\zeta^{j+k}.
   \end{align*}
   Consequently,
   \[
      \Phi(F-\mathbf{1}\otimes\Delta F)=
      \sum_{j,k=-\infty}^\infty\widehat{F}(j,k)\Phi(\zeta^j\otimes\zeta^k-\mathbf{1}\otimes\zeta^{j+k}),
   \]
   and~\eqref{e937} gives
   \begin{equation}\label{e1844}
      \begin{split}
         \abs{\Phi(F-\mathbf{1}\otimes\Delta F)}
         & \le
         \sum_{j,k=-\infty}^\infty \abs{\widehat{F}(j,k)}
         \abs{\Phi(\zeta^j\otimes\zeta^k-\mathbf{1}\otimes\zeta^{j+k})} \\
         & \le
         \sum_{j,k=-\infty}^\infty \abs{\widehat{F}(j,k)}
         \left[ \norm{\Phi\mid_{\ker\Delta}} 2\sin\tfrac{\pi}{10}+
            60\sqrt{27} \left(1+\sin\tfrac{\pi}{10}\right) \varepsilon\right]\\
         & = 
         \norm{F}
         \left[ \norm{\Phi\mid_{\ker\Delta}} 2\sin\tfrac{\pi}{10}+
            60\sqrt{27}\left(1+\sin\tfrac{\pi}{10}\right)\varepsilon\right] .
      \end{split}
   \end{equation}
   In particular, for each $F\in\ker\Delta$, we have
   \[
      \norm{\Phi(F)} \le
      \norm{F}
      \left[ \norm{\Phi\mid_{\ker\Delta}} 2\sin\tfrac{\pi}{10}+
         60\sqrt{27}\left(1+\sin\tfrac{\pi}{10}\right)\varepsilon \right].
   \]
   Thus
   \[
      \norm{\Phi\mid_{\ker\Delta}} \le
      \norm{\Phi\mid_{\ker\Delta}} 2\sin\tfrac{\pi}{10}+
      60\sqrt{27}\left(1+\sin\tfrac{\pi}{10}\right) \varepsilon,
   \]
   so that
   \[
      \norm{\Phi\mid_{\ker\Delta}} \le
      60\sqrt{27}\,\frac{1+\sin\tfrac{\pi}{10}}{1-2\sin\tfrac{\pi}{10}}\,\varepsilon.
   \]
   Using this estimate in~\eqref{e1844}, we obtain
   \begin{equation*}
      \begin{split}
         \abs{\Phi(F-\mathbf{1}\otimes\Delta F)}
         \le &
         \norm{F} \left[
            60\sqrt{27}\,\frac{1+\sin\tfrac{\pi}{10}}{1-2\sin\tfrac{\pi}{10}}\,\varepsilon
            2\sin\tfrac{\pi}{10}
            +
            60\sqrt{27}\left(1+\sin\tfrac{\pi}{10}\right)\varepsilon\right]
         \\
         = &
         \norm{F}
         60\sqrt{27}\,\frac{1+\sin\tfrac{\pi}{10}}{1-2\sin\tfrac{\pi}{10}}\,\varepsilon
      \end{split}
   \end{equation*}
   for each $F\in A(\mathbb{T}^2)$, which completes the proof.
\end{proof}

\begin{theorem}\label{t1821}
   Let $G$ be a locally compact group.
   Then the Banach algebra $L^1(G)$ is strongly zero product determined and
   \[
      \alpha_{L^1(G)}\le\beta_{L^1(G)}\le
      60\sqrt{27}\,\frac{1+\sin\frac{\pi}{10}}{1-2\sin\frac{\pi}{10}}.
   \]
\end{theorem}

\begin{proof}
   On account of Corollary~\ref{914}, it suffices to prove that $L^1(G)$
   has the strong property $\mathbb{B}$ with
   \begin{equation}\label{1229}
      \beta_{L^1(G)}\le
      60\sqrt{27}\,\frac{1+\sin\tfrac{\pi}{10}}{1-2\sin\tfrac{\pi}{10}},
   \end{equation}
   because $L^1(G)$ has an approximate identity of bound $1$.
   For this purpose set $\varphi\in\mathcal{B}^2(L^1(G),\mathbb{C})$.

   Let $t\in G$, and let $\delta_t$ be the point mass measure at $t$ on $G$.
   We define a contractive homomorphism $T\colon A(\mathbb{T})\to M(G)$ by
   \[
      T(u)=\sum_{k=-\infty}^\infty\widehat{u}(k)\delta_{t^k}
      \quad \forall u\in A(\mathbb{T}).
   \]
   Take $f,h\in L^1(G)$ with $\Vert f\Vert=\Vert h\Vert=1$, and define
   a continuous linear functional
   $\Phi\colon A(\mathbb{T}^2)\to\mathbb{C}$ by
   \[
      \Phi(F)=
      \sum_{(j,k)\in\mathbb{Z}^2}\widehat{F}(j,k)\varphi(f\ast\delta_{t^j},\delta_{t^k}\ast h)
      \quad \forall F\in A(\mathbb{T}^2).
   \]
   Further, if $u,v\in A(\mathbb{T})$, then
   \[
      \Phi(u\otimes v)=
      \sum_{(j,k)\in\mathbb{Z}^2}\widehat{u}(j)\widehat{v}(k)\varphi(f\ast\delta_{t^j},\delta_{t^k}\ast h)=
      \varphi(f\ast T(u),T(v)\ast h);
   \]
   in particular, if  $uv=0$, then
   $(f\ast T(u))\ast (T(v)\ast h)=f\ast T(uv)\ast h=0$,
   and so
   \[
      \abs{\Phi(u\otimes v)} =
      \abs{\varphi(f\ast T(u),T(v)\ast h)} \le
      \abs{\varphi}_{zp} \norm{f\ast T(u)} \norm{T(v)\ast h} \le
      \abs{\varphi}_{zp} \norm{u} \norm{v}.
   \]
   By applying Lemma~\ref{l2} with $F=\zeta\otimes\mathbf{1}$, we see that
   \begin{equation*}
      \abs{\varphi(f\ast\delta_t,h)-\varphi(f,\delta_t\ast h)}
      =
      \abs{\Phi(\zeta\otimes\mathbf{1}-\mathbf{1}\otimes\zeta)}
      \le
      60\sqrt{27}\,\frac{1+\sin\frac{\pi}{10}}{1-2\sin\frac{\pi}{10}}\abs{\varphi}_{zp}.
   \end{equation*}
   We now take $g\in L^1(G)$ with $\Vert g\Vert=1$.
   By multiplying the above inequality by $\vert g(t)\vert$, we arrive at
   \begin{equation}\label{1331}
      \abs{\varphi(g(t) f\ast\delta_t,h)-\varphi(f,g(t)\delta_t\ast h)}
      \le
      60\sqrt{27}\,\frac{1+\sin\frac{\pi}{10}}{1-2\sin\frac{\pi}{10}}\abs{\varphi}_{zp}
      \abs{g(t)}.
   \end{equation}
   Since the convolutions $f\ast g$ and $g\ast h$ can be expressed as
   \begin{equation*}
      \begin{split}
         f\ast g
         &=
         \int_G g(t) f\ast\delta_t \, dt,\\
         g\ast h
         &=
         \int_G g(t)\delta_t\ast h \, dt,
      \end{split}
   \end{equation*}
   where the expressions on the right-hand side are
   considered as Bochner integrals of $L^1(G)$-valued functions of
   $t$,
   it follows that
   \[
      \varphi(f\ast g,h)-\varphi(f,g\ast h)=
      \int_G \left[\varphi(g(t)f\ast\delta_t,h) - \varphi(f,g(t) \delta_t\ast h)\right]\, dt.
   \]
   From~\eqref{1331} we now deduce that
   \begin{equation*}
      \begin{split}
         \abs{\varphi(f\ast g,h)-\varphi(f,g\ast h)} & \le
         \int_G\abs{\varphi(g(t)f\ast\delta_t,h)-\varphi(f,g(t) \delta_t\ast h)} \, dt \\
         & \le
         60\sqrt{27}\,\frac{1+\sin\frac{\pi}{10}}{1-2\sin\frac{\pi}{10}}\abs{\varphi}_{zp}
         \int_G\abs{g(t)}\, dt \\
         & =
         60\sqrt{27}\,\frac{1+\sin\frac{\pi}{10}}{1-2\sin\frac{\pi}{10}}\abs{\varphi}_{zp}   .
      \end{split}
   \end{equation*}
   We thus get
   \[
      \vert\varphi\vert_b\le
      60\sqrt{27}\,\frac{1+\sin\frac{\pi}{10}}{1-2\sin\frac{\pi}{10}}
      \abs{\varphi}_{zp},
   \]
   and~\eqref{1229} is proved.
\end{proof}

\section{Algebras of approximable aperators}

Let $X$ be a Banach space.
Then we write $X^*$ for the dual of $X$,
$\mathcal{B}(X)$ for the Banach algebra of continuous linear operators on $X$,
$\mathcal{F}(X)$ for the two-sided ideal of $\mathcal{B}(X)$ consisting of finite-rank operators,
and $\mathcal{A}(X)$ for the closure of $\mathcal{F}(X)$ in $\mathcal{B}(X)$ with respect to the operator norm.
The identity operators on $X$ and $X^*$ are denoted by $I_X$ and $I_{X^*}$, respectively.
For each $x\in X$ and $\phi\in X^*$,
we define $x\otimes\phi\in\mathcal{F}(X)$ by $(x\otimes\phi)(y) = \phi(y) x$ for each $y\in X$.
A \emph{finite, biorthogonal system} for $X$ is a set
\[
   \bigl\{(x_j,\phi_k):  j,k=1,\ldots,n\bigr\}
\]
with $x_1,\ldots,x_n\in X$ and $\phi_1,\ldots,\phi_n\in X^*$ such that
\[
   \phi_k(x_j)=\delta_{j,k}\quad\forall j,k=1,\ldots,n.
\]
Each such system defines an algebra homomorphism
\[
   \theta\colon\mathbb{M}_n\to\mathcal{F}(X),\quad (a_{j,k})\mapsto\sum_{j,k=1}^na_{j,k}x_j\otimes \phi_k,
\]
where $\mathbb{M}_n$ is the full matrix algebra of order $n$ over $\mathbb{C}$.
The identity matrix is denoted by $I_n$.

The Banach space $X$ is said to have \emph{property} $(\mathbb{A})$ if there is a directed set $\Lambda$ such that, for each $\lambda\in\Lambda$, there exists a finite, biorthogonal system
\[
   \bigl\{(x_j^{\lambda},\phi_k^{\lambda}):  j,k=1,\ldots,n_\lambda\bigr\}
\]
for $X$ with corresponding algebra homomorphism $\theta_\lambda\colon\mathbb{M}_{n_\lambda}\to\mathcal{F}(X)$
such that:
\begin{enumerate}
   \item[(i)]
      $\lim_{\lambda\in\Lambda}\theta_\lambda(I_{n_\lambda})=I_X$ uniformly on the compact subsets of $X$;
   \item[(ii)]
      $\lim_{\lambda\in\Lambda}\theta_\lambda(I_{n_\lambda})^*=I_{X^*}$ uniformly on the compact subsets of $X^*$;
   \item[(iii)]
      for each index $\lambda\in\Lambda$, there is a finite subgroup $G_\lambda$ of
      the group of all invertible $n_\lambda\times n_\lambda$ matrices over $\mathbb{C}$ whose linear
      span is all of $\mathbb{M}_{n_\lambda}$, such that
      \begin{equation}\label{A}
         \sup_{\lambda\in\Lambda}\sup_{t\in G_\lambda}\Vert\theta_\lambda(t)\Vert<\infty.
      \end{equation}
\end{enumerate}
For an exhaustive treatment of this topic (including a variety of interesting examples of spaces with property $(\mathbb{A})$)
we refer to~\cite[Section~3.3]{Runde}.

The notation of the above definition will be standard for the remainder of this section.

\begin{theorem}
   Let $X$ be a Banach space with property $(\mathbb{A})$.
   Then the Banach algebra $\mathcal{A}(X)$ is strongly zero product determined.
   Specifically, if $C$ denotes the supremum in~\eqref{A}, then
   \[
      \tfrac{1}{2}\beta_{\mathcal{A}(X)}\le
      \alpha_{\mathcal{A}(X)}\le
      60\sqrt{27}\,\frac{1+\sin\tfrac{\pi}{10}}{1-2\sin\tfrac{\pi}{10}}\,C^2.
   \]
\end{theorem}

\begin{proof}
   For each $\lambda\in\Lambda$ we define $\Phi_\lambda \colon \ell^1(G_\lambda)\to\mathcal{F}(X)$
   by
   \[
      \Phi_\lambda(f)=\sum_{t\in G_\lambda}f(t)\theta_\lambda(t)\quad\forall f\in \ell^1(G_\lambda).
   \]
   We claim that $\Phi_\lambda$ is an algebra homomorphism.
   It is clear the $\Phi_\lambda$ is a linear map and, for each $f,g\in\ell^1(G_\lambda)$, we have
   \begin{equation*}
      \begin{split}
         \Phi_\lambda(f\ast g) & =
         \sum_{t\in G_\lambda} (f\ast g)(t)\theta_{\lambda}(t) =
         \sum_{t\in G_\lambda}\sum_{s\in G_\lambda}f(s)g(s^{-1}t)\theta_{\lambda}(t)\\
         & =
         \theta_\lambda\left(
         \sum_{t\in G_\lambda}\sum_{s\in G_\lambda}f(s)g(s^{-1}t)t\right)
         =
         \theta_\lambda\left(
         \sum_{s\in G_\lambda}f(s)s \sum_{t\in G_\lambda}g(s^{-1}t)s^{-1}t\right)\\
         & =
         \theta_\lambda\left(\sum_{s\in G_\lambda}f(s)s\sum_{r\in G_\lambda}g(r)r\right)
         =
         \theta_\lambda\left(\sum_{s\in G_\lambda}f(s)s\right)\theta_\lambda\left(\sum_{r\in G_\lambda}g(r)r\right)  \\
         & = \Phi_\lambda(f)\Phi_\lambda(g).
      \end{split}
   \end{equation*}
   Of course, $\Phi_\lambda$ is continuous because $\ell^1(G_\lambda)$ is finite-dimensional, and, further,
   for each $f\in\ell^1(G_\lambda)$, we have
   \[
      \norm{\Phi_\lambda(f)} \le
      \sum_{t\in G_\lambda} \abs{f(t)} \norm{\theta_\lambda(t)} \le
      \sum_{t\in G_\lambda} \abs{f(t)} C=
      C\norm{f}_1.
   \]
   Hence $\Vert\Phi_\lambda\Vert\le C$.

   Let $\varphi\in\mathcal{B}^2(\mathcal{A}(X);\mathbb{C})$.
   Let us prove that
   \begin{equation}\label{1100}
      \abs{\varphi(S\theta_\lambda(t),\theta_\lambda(t^{-1})T)-
         \varphi(S\theta_\lambda(I_{n_\lambda}),\theta_\lambda(I_{n_\lambda})T)} \le
      \beta_{\ell^1(G_\lambda)} C^2 \norm{S} \norm{T} \abs{\varphi}_{zp}
   \end{equation}
   for all $\lambda\in\Lambda$, $S,T\in\mathcal{A}(X)$, and $t\in G_\lambda$.
   For this purpose, take
   $\lambda\in\Lambda$ and $S,T\in\mathcal{A}(X)$,
   and define
   $\varphi_\lambda\colon\ell^1(G_\lambda)\times\ell^1(G_\lambda)\to\mathbb{C}$ by
   \[
      \varphi_\lambda(f,g)=\varphi(S\Phi_\lambda(f),\Phi_\lambda(g)T)\quad\forall f,g\in\ell^1(G_\lambda).
   \]
   Then $\varphi_\lambda$ is continuous and, for each $f,g\in\ell^1(G_\lambda)$
   such that $f\ast g=0$, we have $(S\Phi_\lambda(f))(\Phi_\lambda(g)T)=S(\Phi_\lambda(f\ast g))T=0$
   and therefore
   \[
      \abs{\varphi_\lambda(f,g)} \le
      \abs{\varphi}_{zp} \norm{S\Phi_\lambda(f)} \norm{\Phi_\lambda(g)T} \le
      \abs{\varphi}_{zp} C^2 \norm{S} \norm{T} \norm{f}_1 \norm{g}_1,
   \]
   whence
   \[
      \abs{\varphi_\lambda}_{zp}\le C^2 \norm{S} \norm{T} \abs{\varphi}_{zp}.
   \]
   For each $t\in G_\lambda$, we have
   \begin{equation*}
      \begin{split}
         \abs{\varphi_\lambda(\delta_t,\delta_{t^{-1}})-\varphi_\lambda(\delta_{I_{n_\lambda}},\delta_{I_{n_\lambda}})}
         & =
         \abs{\varphi_\lambda(\delta_{I_{n_\lambda}}\ast\delta_t,\delta_{t^{-1}})-
            \varphi_\lambda(\delta_{I_{n_\lambda}},\delta_t\ast\delta_{t^{-1}})}\\
         &\le
         \abs{\varphi_\lambda}_b\le
         \beta_{\ell^1(G_\lambda)}\abs{\varphi_\lambda}_{zp} \leq
         \beta_{\ell^1(G_\lambda)}C^2 \norm{S} \norm{T} \abs{\varphi}_{zp},
      \end{split}
   \end{equation*}
   which gives~\eqref{1100}.

   The projective tensor product $\mathcal{A}(X)\widehat{\otimes}\mathcal{A}(X)$
   becomes a Banach $\mathcal{A}(X)$-bimodule for the products defined by
   \[
      R\cdot(S\otimes T)=(RS)\otimes T,\quad (S\otimes T)\cdot R=S\otimes (TR)\quad \forall R,S,T\in\mathcal{A}(X).
   \]
   We define a bounded linear map $\widehat{\varphi}\colon\mathcal{A}(X)\widehat{\otimes}\mathcal{A}(X)\to\mathbb{C}$
   through
   \[
      \widehat{\varphi}(S\otimes T)=\varphi(S,T)\quad\forall S,T\in\mathcal{A}(X).
   \]
   For each $\lambda\in\Lambda$, set $P_\lambda=\theta_\lambda(I_{n_\lambda})$
   and
   \[
      D_\lambda=
      \frac{1}{\vert G_\lambda\vert}\sum_{t\in G_\lambda}\theta_\lambda(t)\otimes\theta_\lambda(t^{-1}).
   \]
   Then
   $(P_\lambda)_{\lambda\in\Lambda}$ is a bounded approximate identity for $\mathcal{A}(X)$
   and $(D_\lambda)_{\lambda\in\Lambda}$ is an approximate diagonal for $\mathcal{A}(X)$ (see~\cite[Theorem~3.3.9]{Runde}),
   so that
   $\left(\norm{S\cdot D_\lambda-D_\lambda\cdot S}\right)_{\lambda\in\Lambda}\to 0$ for each $S\in\mathcal{A}(X)$.

   For each $\lambda\in\Lambda$ and $S,T\in\mathcal{A}(X)$,~\eqref{1100} shows that
   \begin{gather*}
      \abs{
         \widehat{\varphi}(S\cdot D_\lambda\cdot T)-\varphi(SP_\lambda,P_\lambda T)} \\
      =
      \left\vert
      \frac{1}{\vert G_\lambda\vert}
      \sum_{t\in G_\lambda}\left[
         \varphi(S\theta_\lambda(t),\theta_\lambda(t^{-1})T)-
         \varphi(S\theta_\lambda(I_{n_\lambda}),\theta_\lambda(I_{n_\lambda})T)\right]\right\vert \\
      \le
      \beta_{\ell^1(G_\lambda)}C^2 \norm{S} \norm{T} \abs{\varphi}_{zp}
   \end{gather*}
   and Theorem~\ref{t1821} then gives
   \begin{equation}\label{1131}
      \abs{\widehat{\varphi}(S\cdot D_\lambda\cdot T)-\varphi(SP_\lambda,P_\lambda T)}
      \le
      60\sqrt{27} \, \frac{1+\sin\frac{\pi}{10}}{1-2\sin\frac{\pi}{10}} \,
      C^2 \norm{S} \norm{T} \abs{\varphi}_{zp}.
   \end{equation}

   For each $\lambda\in\Lambda$, define $\xi_\lambda\in\mathcal{A}(X)^*$ by
   \[
      \xi_\lambda(T)=\widehat{\varphi}(D_\lambda\cdot T)\quad\forall T\in\mathcal{A}(X).
   \]
   Note that
   \[
      \norm{\xi_\lambda} \le
      \norm{\widehat{\varphi}} \norm{D_\lambda} \le
      \norm{\varphi} C^2
      \quad\forall\lambda\in\Lambda
   \]
   and therefore $(\xi_\lambda)_{\lambda\in\Lambda}$ is a bounded net in $\mathcal{A}(X)^*$.
   By the Banach--Alaoglu theorem the net $(\xi_\lambda)_{\lambda\in\Lambda}$ has a weak$^*$-accumulation point,
   say $\xi$, in $\mathcal{A}(X)^*$. Take a subnet $(\xi_\nu)_{\nu\in N}$ of $(\xi_\lambda)_{\lambda\in\Lambda}$
   such that w$^*$-$\lim_{\nu\in N}\xi_\nu=\xi$.
   Take $S,T\in\mathcal{A}(X)$.
   For each $\nu\in N$, we have
   \[
      \varphi(SP_\nu,P_\nu T)-\xi_\lambda(ST)=
      \varphi(SP_\nu,P_\nu T)-\widehat{\varphi}(S\cdot D_\nu\cdot T)+
      \widehat{\varphi}\bigl((S\cdot D_\nu-D_\nu\cdot S)\cdot T\bigr)
   \]
   so that~\eqref{1131} gives
   \begin{equation*}
      \begin{split}
         \abs{\varphi(SP_\nu,P_\nu T)-\xi_\lambda(ST)} & \le
         60\sqrt{27}\,\frac{1+\sin\frac{\pi}{10}}{1-2\sin\frac{\pi}{10}}\,
         C^2 \norm{S} \norm{T} \abs{\varphi}_{zp} \\
         & \quad {}+ \norm{\varphi} \norm{S\cdot D_\nu-D_\nu\cdot S} \norm{T}.
      \end{split}
   \end{equation*}
   Taking limits on both sides of the above inequality,
   and using that $(SP_\nu)_{\nu\in N}\to S$, $(P_\nu T)_{\nu\in N}\to T$, and $(\Vert S\cdot D_\nu-D_\nu\cdot S\Vert)_{\nu\in N}\to 0$, we see that
   \[
      \abs{\varphi(S,T)-\xi(ST)} \le
      60\sqrt{27} \, \frac{1+\sin\frac{\pi}{10}}{1-2\sin\frac{\pi}{10}}\,
      C^2 \norm{S} \norm{T} \abs{\varphi}_{zp}.
   \]
   We thus get
   \[
      \dist \bigl(\varphi,\mathcal{B}^2_\pi(\mathcal{A}(X);\mathbb{C})\bigr)\le
      60\sqrt{27}\,\frac{1+\sin\tfrac{\pi}{10}}{1-2\sin\tfrac{\pi}{10}} \,
      C^2 \abs{\varphi}_{zp},
   \]
   which proves the theorem.
\end{proof}

\end{document}